\documentclass[11pt,reqno]{amsart}

\usepackage{fullpage, amsfonts,amsmath,amscd,amssymb,youngtab}
\usepackage{amssymb, amsbsy, amsthm, amsmath, amstext, amsopn, verbatim,
multicol}
\usepackage[all]{xy}
\usepackage{graphicx}
\usepackage{calc}
\usepackage{ifthen}

\input xypic

\theoremstyle{plain}
\newtheorem*{lem}{Lemma}
\newtheorem*{prop}{Proposition}
\newtheorem*{cor}{Corollary}
\newtheorem*{thm}{Theorem}

\theoremstyle{Definition}

\newtheorem*{defn}{Definition}

\newcommand{\bdm}{\begin{displaymath}}
\newcommand{\edm}{\end{displaymath}}

\newcommand{\Id}{\mathrm{Id}}
\newcommand{\Ima}{\mathrm{Im}}

\newcommand{\Irr}{\mathrm{Irr}}

\newcommand{\Z}{\mathbb{Z}}

\newcommand{\C}{\mathbb{C}}

\newcommand{\blambda}{\boldsymbol{\lambda}}
\newcommand{\bmu}{\boldsymbol{\mu}}
\renewcommand{\leq}{\leqslant}
\renewcommand{\geq}{\geqslant}

\newcommand{\gr}{\mathrm{gr}}

\newcommand{\rat}{\overline{H}_{c}}

\newcommand{\Res}{\mathrm{Res}}

\newcommand{\x}{\scriptstyle{1}}
\newcommand{\xx}{\scriptstyle{2}}
\newcommand{\xxx}{\scriptstyle{3}}
\newcommand{\xxxx}{\scriptstyle{4}}
\newcommand{\xxxxx}{\scriptstyle{5}}
\newcommand{\xxxxxx}{\scriptstyle{6}}
\newcommand{\xxxxxxx}{\scriptstyle{7}}
\newcommand{\xxxxxxxx}{\scriptstyle{8}}
\newcommand{\xxxxxxxxx}{\scriptstyle{9}}
\newcommand{\xxxxxxxxxx}{\scriptstyle{10}}

\let\mc\mathcal

\newcommand{\ZZ}{\mathbb{Z}}

\title{Blocks of restricted rational Cherednik algebras for $G(m,d,n)$}
\author{Maurizio Martino}

\address{Mathematisches Institut, Universit\"at Bonn, Endenicher Allee 60, Bonn 53115, Germany}

\email{mmartino@math.uni-bonn.de}

\begin{document}

\begin{abstract}
We study the Dunkl-Opdam subalgebra of the rational Cherednik
algebra for wreath products at $t=0$, and use this to describe the
block decomposition of restricted rational Cherednik algebras for $G(m,d,n)$.
\end{abstract}

\maketitle

\section{Introduction}

\subsection{} Restricted rational Cherednik algebras have been a topic of recent interest due to
connections with cells and families for cyclotomic Hecke algebras, \cite{GM} and \cite{M}.
This paper is an attempt to understand better the Cherednik side of the picture. In particular,
we extend the description of the blocks of the restricted algebras to arbitrary parameters. The
case of rational parameters was dealt with in \cite{G2} (see \cite[Theorem 2.5]{GM}), and the combinatorics reinterpreted in \cite[Theorem 3.13]{M}.
Key to our arguments is the Dunkl-Opdam subalgebra of the rational Cherednik algebra, which we study in the $t=0$ situation.
This algebra has already been used to great effect in the study of rational Cherednik algebras in the $t \neq 0$ case, see for example \cite{DGr}.

\subsection{} Let us decribe our results in a little more detail,
for notation see Sections \ref{RCA} and \ref{DOalgebra}. Let
$W=G(m,1,n)$ be the wreath product of the symmetric group with the
cyclic group of order $m$. Corresponding to tuples $(t,\kappa,c_1,
\dots c_{m-1}) \in \C^{m+1}$, Etingof and Ginzburg defined in
\cite{EG} a flat family of algebras $H_{t,c}$ called rational
Cherednik algebras. When $t=0$, these algebras, which are denoted
$H_c$, are finite-dimensional modules over their centres. There
exists a particularly interesting finite dimensional quotient
algebra $\rat$ of $H_c$ called the restricted rational Cherednik
algebra. The main focus of this paper is to determine the block
structure of $\rat$.

\subsection{}\label{mainthmstatement} The algebras $\rat$ have many interesting properties. There exist
baby Verma modules $\overline{\Delta}_c(E)$ for each irreducible
$W$-module $E$. These are indecomposable modules whose simple heads
yield all simple $\rat$-modules. It is well-known that the
irreducible $W$-modules are parametrised by $m$-multipartitions of
$n$. Given a multipartition $\blambda=(\lambda^0, \dots
,\lambda^{m-1})$, we denote by $\overline{\Delta}_c(\blambda)$ the
corresponding baby Verma module. Our main result, Theorem
\ref{wreathblocks}, is that one can use the combinatorics of
multipartitions to describe the block decomposition of $\rat$. The
numbers $a_i$ are defined in \ref{Zchar} and the polynomials
$\Res_{\lambda^i}(x)$ denote residue of $\lambda^i$, which is
defined in \ref{combinatorics}.

\begin{thm}\label{wreathblock}
Let $\blambda$ and $\bmu$ be $m$-multipartitions of $n$. The baby
Verma modules $\overline{\Delta}_c(\blambda),
\overline{\Delta}_c(\bmu)$ lie in the same block if and only if
$\sum_{i=0}^{m-1} x^{a_i} \Res_{\lambda^i}(x^{-\kappa}) =
\sum_{i=0}^{m-1} x^{a_i}\Res_{\mu^i}(x^{-\kappa})$.
\end{thm}

\subsection{}
It is interesting to note the method of proof. An important role is
played by the Dunkl-Opdam subalgebra, which is generated by the
elements $z_i$ from \ref{HeckeRelations}. This is a commutative
subalgebra of $H_c$, and a crucial point is that the subalgebra
$\mathfrak{Z}_c$ of symmetric polynomials in the $z_i$ is central in
$H_c$, Theorem \ref{centralthm}. This is proved in Section
\ref{proofofcentral} by direct calculation; a proof using the theory
of generalised Jack polynomials was kindly passed on to the author
by Stephen Griffeth and we sketch this argument in \ref{Jackproof}.
It is shown in Theorem \ref{wreathblocks} that the action of
$\mathfrak{Z}_c$ on baby Verma modules determines the blocks for
$\rat$. Evaluating the eigenvalues of $\mathfrak{Z}_c$ on baby Verma
modules then yields the block description of Theorem
\ref{wreathblock}. We conclude the paper by applying the results
from \cite{B} to describe the block decomposition of restricted
Cherednik algebras for the normal subgroups $G(m,d,n)$, Theorem
\ref{normalblocks}.

\subsection{Relation to Hecke algebras and further questions}The
combinatorics governing the blocks of $\rat$ are conjectured to be
related to cells at unequal parameters, \cite{GM}, and to Rouquier
families for cyclotomic Hecke algebras, \cite{M} (where our $a_i$
are written $\overline{m}_i$). In both these cases the parameters
(more precisely the $H_i$s from \ref{wreathcherednik}) need to be
rational in order to be able to interpret the corresponding Hecke
algebras. Our results hold for all parameters, and we note here
another interpretation of the resulting combinatorics which holds
also in non-rational cases. Let us assume that $\kappa \neq 0$, so
that without loss of generality we can set $\kappa = -1$, see
(\ref{shiftiso}). Let $H_n^a$ be the degenerate cyclotomic Hecke
algebra for $S_n$ associated to the parameter $a$ from
\ref{wreathblock}, see \cite{Br}. Then the blocks of $H_n^a$ are
precisely determined by the formula in Theorem \ref{wreathblock},
\cite[Lemma 4.2]{Br}. It would be interesting to know whether $\rat$
is more closely related to $H_n^a$, or to the associated parabolic
category $\mathcal{O}$.

It is natural to ask whether analogues of the Dunkl-Opdam subalgebra
exist for the exceptional complex reflection groups, and whether the
methods presented here could be extended to those cases. These
techniques should also be applicable to studying blocks for
restricted Cherednik algebras in characteristic $p>0$ and at $t\neq
0$. We will return to these questions in future work.

\subsection{Acknowledgements}
We would like to thank Stephen Griffeth for informing us of an
alternative proof of Theorem \ref{centralthm} and for allowing us to
reproduce it in \ref{Jackproof}. We thank Gwyn Bellamy for useful
comments. This work was supported by the SFB/TR 45 ``Periods, Moduli
Spaces and Arithmetic of Algebraic Varieties" of the DFG (German
Research Foundation).

\section{Rational Cherednik algebras}\label{RCA}

\subsection{}\label{notation} Let $W$ be a complex reflection group and $V$ its
reflection representation over $\C$. Let $\mathcal{S} \subset W$
denote the set of complex reflections in $W$. For $s\in \mathcal{S}$
let $H_s \subset V$ denote the reflection hyperplane
$\mathrm{Fix}_s(V)$, and let $W_s$ denote the pointwise stabiliser
of $H_s$. For every $s \in \mathcal{S}$ we choose $v_s \in V$ such
that $\C v_s$ is a $W_s$-stable complement to $H_s$, and choose also
a linear form $\alpha_s \in V^*$ with kernel $H_s$. Let $<\ ,\ >$
denote the natural pairing of $V^*$ with $V$. 


\subsection{Rational Cherednik algebras}\label{PBW} We introduce formal parameters
$\underline{t}, \underline{c}_{s}$ for $s\in \mathcal{S}$, where we set $\underline{c}_{s} = \underline{c}_{t}$ if the
conjugacy classes of $s$ and $t$ are equal. Let ${R}$ be the polynomial ring
$\C[\underline{t}, \underline{c}_{s}]_{{s} \in \mathcal{S}}$. Let $T_{R}(V \oplus V^*)$ and $S_{R}(V \oplus V^*)$ denote
the tensor and symmetric algebras, respectively, of $V \oplus V^*$ over ${R}$.

\begin{defn} The \textit{rational Cherednik algebra} $H_{{R}}$
is the quotient of the smash product $T_{R}(V \oplus
V^*)\rtimes W$ by the relations:
\begin{align*} [x,x']=0,\ [y,y']=0\ \mathrm{and}\ [y,x] = \underline{t}<x,y> - \sum_{s \in
\mathcal{S}} \underline{c}_{s}\frac{<\alpha_s,y><x,v_s>}{<\alpha_s,v_s>} s
\end{align*}
for $y,y'\in V$ and $x,x'\in V^*$.
\end{defn}

Let $t\in \C$ and let $c: \mathcal{S} \to \C$ be a $W$-invariant
function. We write $c_s$ for the value $c(s)$, and identify $c$ with
an element in $\C^{|\mathcal{S}/W|}$. Let $m_{t,c}$ be the maximal
ideal of ${R} = \C[\C \times \C^{|\mathcal{S}/W|}]$ corresponding to
$(t,c)$. The rational Cherednik algebra specialised at $(t,c)$ is
the algebra $H_{t,c} := R/m_{t,c} \otimes_R H_R$.

The first important property about $H_R$ is the following PBW theorem.

\begin{thm}
The multiplication map $S_{R}(V) \otimes_{R} {R}W \otimes_{R} S_{R}(V^*) \to H_{R}$ is an isomorphism
of $R$-modules.
\end{thm}
\begin{proof}
The analogous property is proved for all $H_{t,c}$ in \cite[Theorem 1.3]{EG}.
The multiplication map over $R$ is clearly surjective.
If $x \in S_{R}(V) \otimes_{R} {R}W \otimes_{R} S_{R}(V^*)$
lies in the kernel, then there exists a $(t,c)$ such that the
specialisation $x_{t,c}$ of $x$ at $(t,c)$ is nonzero,
and furthermore $x_{t,c}$ lies in the kernel of the
multiplication map $S_{\C}(V) \otimes_{\C} {\C}W \otimes_{\C} S_{\C}(V^*) \to H_{t,c}$,
a contradiction.\end{proof}

We can filter $H_{R}$ by the putting $V$ and $V^*$ in degree $1$, and putting ${R} W$ in degree $0$, and we denote by $\gr H_{R}$ the corresponding associated graded algebra.
A consequence of the PBW theorem is the following.

\begin{cor}
There is an isomorphism of ${R}$-algebras $\gr H_{R} \cong S_{R}(V \oplus V^*) \rtimes W$.
\end{cor}

There exists a $\Z$-grading on $H_R$ by putting $V$ and $V^*$ in
degree $1$ and $-1$, respectively, and putting ${R} W$ in degree
$0$. This is obtained by considering an algebraic $\C^*$-action on
$H_R$ with the corresponding weights. In particular, it follows that
the filtered pieces of $H_R$ are also $\Z$-graded.

\subsection{Restricted rational Cherednik algebras}\label{RRCA}
Let ${c} =(0,c)$ be a parameter and let $H_{c} = H_{(0,c)}$. By
\cite[Proposition 3.6]{G1}, there is an algebra monomorphism \[ A :=
(S_{\C} V)^W \otimes_{\C} (S_{\C} V^*)^W \hookrightarrow H_{c},\]
whose image lies in the centre of $H_{c}$. Let $A^+$ denote the
ideal in $A$ generated by polynomials with zero constant term. The
quotient algebra $\overline{H}_{c} := H_{c}/<A^+>$ is called the
\textit{restricted rational Cherednik algebra}.

By Theorem \ref{PBW}, $H_{c}$ is a free $A$-module, and
$\overline{H}_{c} \cong (S_{\C} V)^{co\ W} \otimes_{\C} {\C}W
\otimes_{\C} (S_{\C} (V^*))^{co\ W}$ as vector spaces (here $(S_{\C}
V)^{co\ W}$ denotes the coinvariant ring). The $\Z$-grading on $H_R$
induces a $\Z$-grading on $H_c$ and also on $\overline{H}_c$, since
$<A^+>$ is a graded ideal.

For any $a \in \C^*$, it is easy to show that there is an
isomorphism of algebras $H_c \cong H_{ac}$, which induces
isomorphisms
\begin{align}\label{shiftiso} \overline{H}_c \cong \overline{H}_{ac}.
\end{align}

\subsection{Baby Verma modules}\label{babyVs} Let $\Irr \C W$ denote the complex irreducible representations of $W$.
Let $E \in \Irr \C W$. We can construct the corresponding baby Verma
module as in \cite{G1}: it is the induced module $\overline{H}_{c}
\otimes_{(S_{\C} V)^{co\ W} \rtimes W} E$, where the action of an
element $pg \in (S_{\C} V)^{co\ W} \rtimes W$ on $E$ is given by
$p(0)g$. Denote this module by $\overline{\Delta}_c(E)$.

\begin{prop}\cite[Proposition 4.3]{G1}
The baby Verma modules $\overline{\Delta}_c(E)$ are indecomposable
$\overline{H}_{c}$-modules. They have simple heads $L_c(E)$, and the
set $\{ L_c(E): E\in \Irr \C W\}$ is a complete set of pairwise
nonisomorphic simple modules for $\overline{H}_{c}$.
\end{prop}

\section{Cherednik algebras for $G(m,1,n)$}\label{DOalgebra}

\subsection{}\label{wreathcherednik} Let $m,n$ be positive integers. Let $C_m$ be the cyclic group of order $m$. We fix a generator $g \in C_m$ and let $\eta \in \C$ be a
primitive $m$th root of unity. Let $s_{ij}\in S_n$ denote the
transposition which swaps $i$ and $j$. We denote by $s_j$ the simple
transposition swapping $j$ and $j+1$. The group $W=G(m,1,n)$ is the
semidirect product $S_n \ltimes (C_m)^n$. We write $g_i^l$ for
the element $(1, \dots ,g^l, \dots 1) \in G(m,1,n)$ with $g^l$ in
the $i$th place. Let $V$ be the reflection representation of
$G(m,1,n)$: we fix a basis $\{ y_1, \dots ,y_n\}$ of $V$ so that
\begin{align*}g_i(y_j) = \begin{cases} \eta y_j\ \mathrm{if}\ i=j\\ y_j\ \mathrm{otherwise} \end{cases} \mathrm{and}\ \sigma(y_j) =
y_{\sigma(j)},\end{align*} for all $i, j$ and all $\sigma \in S_n$.
Let $\{x_1, \dots ,x_n\} \in V^*$ be the dual basis.

The conjugacy classes of reflections in $W$ are given
by \begin{align}\label{conj1} \{s_{ij} g_i^{-l} g_j^l: 0 \leq l \leq m-1\ \mathrm{and}\ i\neq j\},
\end{align} and, for each $1 \leq l \leq m-1$, \begin{align}\label{conj2} \{
g_j^l: 1 \leq j \leq n\}.\end{align} We fix formal variables
$(\underline{t}; \underline{\kappa}, \underline{c}_1, \dots
,\underline{c}_{m-1})$, where $\underline{\kappa}$ corresponds to
(\ref{conj1}) and the $\underline{c}_l$ correspond to the classes in (\ref{conj2}).
Let $R = \C[\underline{t}, \underline{\kappa},
\underline{c}_1, \dots , \underline{c}_{m-1}]$. In the case
$W=G(m,1,n)$, Definition \ref{PBW} becomes the following.

\begin{defn}The rational Cherednik algebra for $G(m,1,n)$, $H_{R} = H_{R}(G(m,1,n))$, is the quotient of the smash
product $T_{R}(V \oplus V^*) \rtimes W$ by the
relations:
\begin{align*} &[x_i , x_j] = 0,\ \ [y_i, y_j]=0, \\ &[y_i , x_i] = \underline{t} -
\underline{\kappa} \sum_{l=0}^{m-1} \sum_{j \neq i} s_{ij} g_i^{-l} g_j^l -
\sum_{l=1}^{m-1} \underline{c}_l(1-\eta^{-l}) g_i^l,\\ &[y_i, x_j] =
\underline{\kappa}\sum_{l=0}^{m-1} \eta^{-l} s_{ij} g_i^{-l} g_j^{l}.
\end{align*}
\end{defn}


\medskip

As in \ref{RRCA}, if we are given a tuple $c=(0; \kappa, c_1, \dots
,c_{m-1}) \in \C^{m+1}$, then we can specialise $H_{{R}}$ to the
algebra $H_c$. We will also use the formal parameters
$(\underline{t}, \underline{h}, \underline{H}_0, \dots ,
\underline{H}_{m-1})$ from \cite{G2}. Thus, $\underline{H}_0 + \dots
+ \underline{H}_{m-1} = 0$, and the $\underline{h}$,
$\underline{H}_i$ are related to the parameters above via
\begin{align}\underline{h} = - \underline{\kappa}\ \mathrm{and}\
-\underline{c}_l(1-\eta^{-l}) = \sum_{j=0}^{m-1}
\eta^{-lj}\underline{H}_j.\label{parameters} \end{align} We will
denote by $(0,h,H_0, \dots ,H_{m-1})$ the image of the
$\underline{t}, \underline{h}, \underline{H}_i$ in $R/m_{(0,c)}$ -
these then become another set of parameters for $H_c$.

\subsection{The Dunkl-Opdam operators}\label{DOops} For all $1 \leq i \leq n$, we
define elements in $H_{R}$: \begin{align} z_i &= y_i x_i -
\frac{1}{2}\underline{t} + \underline{\kappa} \sum_{l=0}^{m-1}
\sum_{1 \leq j < i} s_{ij} g_i^l g_j^{-l} - \sum_{l=1}^{m-1}
\underline{c}_l \eta^{-l} g_i^l\label{DO1}\\ &= x_i y_i +
\frac{1}{2}\underline{t} - \underline{\kappa} \sum_{l=0}^{m-1}
\sum_{i < j} s_{ij} g_i^{l} g_j^{-l} - \sum_{l=1}^{m-1}
\underline{c}_l g_i^l \label{DO2}.\end{align}

\medskip

\noindent \textbf{Remark.} The elements $z_i$ have already appeared
in \cite{DO} and \cite{Gr1} for the rational Cherednik algebra
specialised at $\underline{t} = 1$, where their action on Verma
modules is studied. Our definition differs from these because we
have added the terms $- \frac{1}{2}\underline{t}$ and $-
\sum_{l=1}^{m-1} \underline{c}_l \eta^{-l} g_i^l$ in (\ref{DO1}).
This is important for Theorem \ref{centralthm}, but using this
definition changes little in the $\underline{t} = 1$ applications.

\medskip

We define \[ \xi_{ij} = \sum_{l=0}^{m-1}
g_i^l g_j^{-l}.\]

\begin{lem}\label{HeckeRelations} The following relations hold in
$H_{R}$:  \begin{itemize} \item[(a)] $[z_i, z_j] = 0$;
\item[(b)] $z_i g_k = g_k z_i$; \item[(c)] $z_i s_{i} =
s_{i} z_{i+1} - \underline{\kappa} \xi_{i,i+1}$ \item[(d)] $z_i
s_j = s_j z_i$ for $i\neq j,j+1$. \end{itemize}
\end{lem}
\begin{proof}
By \cite[Propositions 4.2 and 4.3]{Gr1}, these hold for the elements
$\tilde{z}_i = z_i + \sum_{l=1}^{m-1} \eta^{-l} \underline{c}_l
g_i^l + \frac{1}{2} \underline{t}$. It is straightforward to check
that the relations (a)-(d) above can be deduced from this.
\end{proof}

By Theorem \ref{PBW}, the algebra $\mathfrak{t}_R$ generated over
$R$ by the $z_i$ and $W$ is isomorphic to a graded Hecke algebra for
$G(m,1,n)$ as introduced in \cite[$\S$ 5]{RS}. In particular, the
subalgebra generated by the $z_i$s is a polynomial algebra in $n$
variables. It follows from \cite[Lemma 5.1]{Gr1} that the
$R$-subalgebra $\mathfrak{Z}_R$ of symmetric polynomials in the
$z_i$ is central in $\mathfrak{t}_R$, and in particular commutes
with $W$.

\subsection{An involution on $H_R$}\label{duality} We define an automorphism $\psi$ of $H_R$ via:
\begin{align*} \underline{t} \mapsto -\underline{t},\ \underline{\kappa}
\mapsto &-\underline{\kappa},\ \underline{c}_l \mapsto \eta^l
\underline{c}_{-l},
\\ s_i \mapsto s_{n-i},\ g_i \mapsto g_{n-i+1}^{-1},&\
x_i \mapsto y_{n-i+1},\ y_i \mapsto x_{n-i+1}.
\end{align*}

It is an easy check using the relations in \ref{wreathcherednik}
that $\psi$ is indeed a $\C$-algebra homomorphism, and clearly
$\psi^2 = \Id$. Furthermore, by the equality of (\ref{DO1}) and
(\ref{DO2}) it follows that $\psi(z_i) = z_{n-i+1}$ for all $i$.

\subsection{A central theorem}\label{centralthm} We now state one of our main theorems,
the proof of which will occupy Section \ref{proofofcentral}.

\begin{thm}
Let $\mathfrak{S}_r = \sum_{1\leq j_1 < j_2 < \dots < j_r \leq n}
z_{j_1} \dots z_{j_r}$ denote the $r$th symmetric polynomial in the
$z_i$s. Then $[x_1, \mathfrak{S}_r]$ lies in $\underline{t}H_{R}$.
In particular, the element $\mathfrak{S}_r$ specialised at $c=(0,c)$
is central in $H_{c}$.
\end{thm}

Let us explain how the second claim follows from the first. We
assume that $\mathfrak{S}_r$ is specialised at $c$. We know that
$\mathfrak{S}_r$ commutes with $W$, \ref{DOops}, so it remains to
show that $[\mathfrak{S}_r, x_i] = [\mathfrak{S}_r, y_i] = 0$ for
all $i$. From the first part of the theorem we know that
$[\mathfrak{S}_r, x_1]=0$. Thus $s_{1i}[\mathfrak{S}_r, x_1]s_{1i} =
[s_{1i} \mathfrak{S}_r s_{1i}, x_i] = [\mathfrak{S}_r, x_i] =0$. Now
we apply the involution from \ref{duality} to obtain
$[\mathfrak{S}_r, y_{n-i+1}] = \psi([\mathfrak{S}_r, x_i]) = 0$ for
all $i$.

\medskip

\noindent \textbf{Remark.} In the theory of rational Cherednik
algebras, an important role is played by the Euler element
\[ \mathbf{eu} := \sum_{i=1}^n x_i y_i + \frac{n}{2}\underline{t} -
\underline{\kappa} \sum_{i \neq j} \sum_{l=0}^{m-1} s_{ij} g_i^{-l}
g_j^l - \sum_{i=1}^n \sum_{l=1}^{m-1} \underline{c}_l g_i^l.
\]
An easy check shows that the sum $\sum_{i=1}^n z_i$ equals the
$\mathbf{eu}$ in $H_{R}$. It is known that specialising $\bf eu$ at
$c$ produces a central element, see \cite[Section 3.1 (4)]{GGOR}.
This verifies the theorem in the special case $r=1$.

\section{Proof of Theorem \ref{centralthm}}\label{proofofcentral}

\subsection{Identities in the graded Hecke algebra}\label{techcalculations}
We begin our proof of Theorem \ref{centralthm} by listing some
relations in $H_R$. We define for all $i \neq j$, \[ \gamma_{ij} = -
\underline{\kappa} \sum_{l=0}^{m-1} s_{ij} g_i^{-l} g_j^{l}.
\] To simplify notation, let $\gamma_j = \gamma_{1j}$ for all $j
> 1$. The next relations all follow directly (by sometimes fairly lengthy computations) from the Definition \ref{PBW} and Lemma
\ref{HeckeRelations}:

\begin{align}
\gamma_{ij} &= \gamma_{ji} \label{gamma1};\\
\sigma \gamma_{ij} \sigma^{-1} &= \gamma_{\sigma(i), \sigma(j)}\ \mathrm{for\ all}\ \sigma \in S_n \label{gamma2};\\
\gamma_{u} z_v &= \begin{cases} z_u \gamma_u + \sum_{1 < t \leq u}
\gamma_u \gamma_t\ &\mathrm{if}\ v=1,
\\ z_v \gamma_u\ &\mathrm{if}\ u < v ,  \\
(z_1 - \sum_{1 < t \leq u} \gamma_t) \gamma_u\ &\mathrm{if}\ u=v,\\
z_v \gamma_u + \gamma_v \gamma_u - \gamma_u \gamma_v\ &\mathrm{if}\
u > v \neq 1. \label{gamma4}
\end{cases}
\end{align}
The relations (\ref{gamma1})-(\ref{gamma4}) yields the following
identities (we assume that $u,v \neq 1$):
\begin{align}\label{gamma5} \gamma_u (z_v + \gamma_v) = \begin{cases}
(z_v + \gamma_{uv})\gamma_u\ &\mathrm{if}\ 1 \neq u < v,\\ (z_1 - \sum_{1 < t < u} \gamma_t) \gamma_u\ &\mathrm{if}\ u=v,\\
(z_v + \gamma_v)\gamma_u\ &\mathrm{if}\ u
> v; \end{cases} \end{align} \begin{align}\label{gamma6} (z_v +
\gamma_v)\gamma_u = \begin{cases} \gamma_u(z_v + \gamma_{uv})\
&\mathrm{if}\ 1\neq u<v,\\ \gamma_u (z_1
- \sum_{1 < t < u} \gamma_t)\ &\mathrm{if}\ u = v,\\
\gamma_u(z_v + \gamma_v)\ &\mathrm{if}\ u > v.
\end{cases}
\end{align}
Let $1 < k \leq u$. An important consequence of (\ref{gamma6}) for
us will be:
\begin{align}
(z_u + \gamma_u) (\sum_{k < t \leq u} \gamma_t) &= \gamma_u (z_1 -
\sum_{1 < t < u} \gamma_t) + (z_u + \gamma_u) (\sum_{k < t < u}
\gamma_t)\notag \\ &= \gamma_u z_1 -\gamma_u \sum_{1 < t \leq k}
\gamma_t +  \sum_{k < t < u} \gamma_t z_u.\label{useful}
\end{align}

We now record
commutation relations for the $z_i$ in $H_{R}$. Let $1 \leq i,j \leq n$ be integers. Then \begin{align}\label{zcomm} [x_i , z_j] =
\begin{cases} -\underline{t} x_i -
\sum_{1 \leq k < i} x_i \gamma_{ki} - \sum_{i < k \leq n} \gamma_{ik} x_i\ &\mathrm{if}\ i=j,\\
\gamma_{ij} x_i\ &\mathrm{if}\ i < j,\\ x_i \gamma_{ij}\
&\mathrm{if}\ i > j;
\end{cases} \end{align} and \begin{align} [y_i, z_j] = \begin{cases} \underline{t} y_i + \sum_{1 \leq k < i} \gamma_{ki} y_i +
\sum_{i < k \leq n} y_i \gamma_{ik}\ &\mathrm{if}\ i=j,\\ - y_i
\gamma_{ij}\ &\mathrm{if}\ i < j,\\ - \gamma_{ij} y_i\ &\mathrm{if}\
i
> j.
\end{cases} \end{align}

\subsection{} We define elements in the Dunkl-Opdam subalgebra which
will play a significant role in our proof of Theorem
\ref{centralthm}.

\begin{defn} Given integers $1 < j_1 < \dots < j_r \leq n$, let
\[{P}_{j_1, \dots ,j_r} := \sum z_{k_1} \dots z_{k_s} \gamma_{l_1} \dots \gamma_{l_t},\]
where the sum is taken over all $1 < k_1 < \dots < k_s \leq n$, $ 1
< l_1 < \dots < l_t \leq n$ such that $t \geq 1$ and $\{ k_1 \dots
,k_s\} \cup \{ l_1 , \dots , l_t\} = \{ j_1, \dots , j_r\}$. We also
define \[\tilde{P}_{j_1, \dots ,j_r} := \sum z_{k_1} \dots z_{k_s}
\gamma_{l_1} \dots \gamma_{l_t},\] where the sum is taken as above,
except that we allow $t=0$.
\end{defn}

It is clear from the definitions that \begin{align} \tilde{P}_{j_1,
\dots ,j_r} &= {P}_{j_1, \dots ,j_r} + z_{j_1} \dots
z_{j_r}\label{Pver1}.\end{align} Since $\gamma_{j_t} z_{j_s} =
z_{j_s} \gamma_{j_t}$ if $t < s$ by (\ref{gamma4}), it follows that
\begin{align} \tilde{P}_{j_1, \dots ,j_r} = (z_{j_1} + \gamma_{j_1}) \dots
(z_{j_r} + \gamma_{j_r}) \label{tildePeqn}.\end{align} Now combining
(\ref{Pver1}) and (\ref{tildePeqn}) we obtain:
\begin{align}
{P}_{j_1, \dots ,j_r} =
(z_{j_1} + \gamma_{j_1}) P_{j_2, \dots ,j_r} + \gamma_{j_1} z_{j_2} \dots z_{j_r}.\label{Peqn0}
\end{align}

\medskip

\begin{lem}\label{commrelations}Let $1 < j_1 < \dots < j_r \leq n$.
In the algebra $H_{R}$ we have $ [x_1, z_{j_1}z_{j_2}\dots z_{j_r}]
= P_{j_1, \dots ,j_r}x_1$ or equivalently, $ x_1 z_{j_1}z_{j_2}\dots
z_{j_r} = \tilde{P}_{j_1, \dots ,j_r}x_1$.
\end{lem}
\begin{proof}
The proof is by induction on $r$. The case $r=1$ follows from
(\ref{zcomm}). Now by (\ref{zcomm}), induction and (\ref{Peqn0}),
\begin{align*} [x_1, z_{j_1}z_{j_2}\dots z_{j_r}] &= [x_1 ,z_{j_1}]
z_{j_2}\dots z_{j_r} + z_{j_1} [x_1, z_{j_2}\dots z_{j_r}]\\ &=
\gamma_{j_1} x_1 z_{j_2} \dots z_{j_r} + z_{j_1} P_{j_2, \dots ,j_r}
x_1\\ &= \gamma_{j_1} [x_1, z_{j_2} \dots z_{j_r}] + \gamma_{j_1}
z_{j_2} \dots z_{j_r} x_1 + z_{j_1} P_{j_2, \dots ,j_r} x_1\\ &=
\gamma_{j_1} P_{j_2, \dots ,j_r} x_1 + \gamma_{j_1} z_{j_2} \dots
z_{j_r} x_1 + z_{j_1} P_{j_2, \dots ,j_r} x_1\\ &= P_{j_1, \dots
,j_r}x_1.
\end{align*} The equivalence of the equations is a consequence of (\ref{Pver1}).
\end{proof}

\subsection{Key calculation} We require one further technical
result.

\begin{prop}\label{greatprop}
Keep the notation from \ref{commrelations}. Let $1 \leq k < j_1$.
Then
\begin{align*} \sum_{k < j_1 < \dots < j_r \leq n} (z_1 - \sum_{k <
t \leq n} \gamma_t) \tilde{P}_{j_1, \dots ,j_r} = \sum_{k < j_1 <
\dots < j_r \leq n} z_1 z_{j_1}  \dots z_{j_{r}} - \sum_{k < j_1 <
\dots < j_r < j_{r+1} \leq n} {P}_{j_1, \dots, j_{r+1}}\\ + \sum_{k
< j_1 < \dots < j_r \leq n}\sum_{1 < t \leq k} \gamma_t P_{j_1,
\dots ,j_r}.\end{align*}
\end{prop}
\begin{proof}
We prove this by induction on $r$. We begin with the case $r=1$. Let
$1\leq k < j<n$. Then by (\ref{gamma5}),
\begin{align}
(z_1 - \sum_{k < t \leq n} \gamma_t)(z_j + \gamma_j) &= z_1 z_j +
z_1 \gamma_j - \sum_{k < t < j} \gamma_t (z_j + \gamma_j) - z_1
\gamma_j + \sum_{1 < t <j} \gamma_t \gamma_j - \sum_{j < t \leq n}
(z_j + \gamma_j)\gamma_t \notag \\ &= z_1 z_j - \sum_{k < t < j}
\gamma_t z_j + \sum_{1 < t \leq k} \gamma_t \gamma_j - \sum_{j < t
\leq n} (z_j + \gamma_j)\gamma_t.\label{induction1}
\end{align}
Thus, \begin{align} \sum_{k < j \leq n} (z_1 - \sum_{k < t \leq n}
\gamma_t)&(z_j + \gamma_j) \notag \\&= \sum_{k < j \leq n} z_1 z_j -
\sum_{k < j \leq n} \left( \sum_{k < t < j} \gamma_t z_j + \sum_{j <
t \leq n} (z_j + \gamma_j)\gamma_t\right) + \sum_{k < j \leq
n}\sum_{1 < t \leq k} \gamma_t \gamma_j \notag \\ &= \sum_{k < j
\leq n} z_1 z_j - \sum_{k < j_1 < j_2 \leq n} P_{j_1, j_2} + \sum_{k
< j \leq n}\sum_{1 < t \leq k} \gamma_t P_j,\label{induction1a}
\end{align} where (\ref{induction1a}) follows from (\ref{Peqn0}).

For the induction step, we first apply (\ref{induction1}) to obtain
\begin{align*}
\sum_{k < j_1 < \dots < j_r \leq n} (z_1 &- \sum_{k < t \leq n}
\gamma_t) \tilde{P}_{j_1, \dots ,j_r}\\ &= \sum_{k < j_1 < \dots <
j_r \leq n} (z_1 - \sum_{k < t \leq n} \gamma_t)(z_{j_1} +
\gamma_{j_1}) \tilde{P}_{j_2, \dots ,j_r}\\ &= \sum_{k < j_1 < \dots
< j_r \leq n} \big( (z_{j_1} + \gamma_{j_1}) (z_1 - \sum_{j_1 < t
\leq n} \gamma_t) - \gamma_{j_1}z_1 - \sum_{k < t < j_1} \gamma_t
z_{j_1} + \sum_{1 < t \leq k} \gamma_t \gamma_{j_1} \big)
\tilde{P}_{j_2, \dots ,j_r}.
\end{align*} By induction, this equals \begin{align*}  \sum_{k < j_1 < \dots <
j_r \leq n} ( - \gamma_{j_1}z_1 &- \sum_{k < t < j_1} \gamma_t
z_{j_1} + \sum_{1 < t \leq k} \gamma_t \gamma_{j_1} )
\tilde{P}_{j_2, \dots ,j_r}
\\ &+ \sum_{k < j_1 < \dots < j_r \leq
n} (z_{j_1} + \gamma_{j_1}) \big( z_1 z_{j_2} \dots z_{j_{r}} -
\sum_{j_r < j_{r+1} \leq n} {P}_{j_2, \dots, j_{r+1}} + \sum_{1 < t
\leq j_1} \gamma_t P_{j_2, \dots ,j_r} \big).\end{align*} By
applying (\ref{Pver1}) and rearranging terms we obtain
\begin{align}
\sum_{k < j_1 < \dots < j_r \leq n} ( - \gamma_{j_1}z_1 - \sum_{k <
t < j_1} \gamma_t z_{j_1} &+ \sum_{1 < t \leq k} \gamma_t
\gamma_{j_1} + (z_{j_1} + \gamma_{j_1})z_1) z_{j_2} \dots z_{j_r} \notag \\
+ \sum_{k < j_1 < \dots < j_r \leq n} ( - \gamma_{j_1}z_1 - \sum_{k
< t < j_1} \gamma_t z_{j_1} + \sum_{1 < t \leq k} \gamma_t
\gamma_{j_1} ) &{P}_{j_2, \dots ,j_r} \notag  - \sum_{k < j_1 <
\dots < j_r < j_{r+1} \leq n} (z_{j_1} + \gamma_{j_1}) {P}_{j_2,
\dots, j_{r+1}}\\ +  \sum_{k < j_1 < \dots < j_r \leq
n}(\gamma_{j_1} z_1 &+ \sum_{1 < t < {j_1}} \gamma_tz_{j_1}) P_{j_2,
\dots ,j_r} \label{inductionmain} ,
\end{align} where line (\ref{inductionmain})
follows from (\ref{useful}). We now cancel terms to obtain
\begin{align*}
\sum_{k < j_1 < \dots < j_r \leq n} ( - \sum_{k < t < j_1} \gamma_t
z_{j_1} + \sum_{1 < t \leq k} \gamma_t \gamma_{j_1} + z_{j_1}z_1)
z_{j_2}& \dots z_{j_r}- \sum_{k < j_1  < \dots < j_r < j_{r+1} \leq
n}
(z_{j_1} + \gamma_{j_1}) {P}_{j_2, \dots, j_{r+1}}\\
+ \sum_{k < j_1 < \dots < j_r \leq n} (\sum_{1 < t \leq k}
\gamma_tz_{j_1} &+ \sum_{1 < t \leq k} \gamma_t \gamma_{j_1} )
{P}_{j_2, \dots ,j_r}.
\end{align*}
By rearranging terms and rewriting the summation indices, we obtain
\begin{align*}
\sum_{k < j_1 < \dots < j_r \leq n} z_{1}z_{j_1}& z_{j_2} \dots
z_{j_r}  - \sum_{k < j_1 < \dots < j_r < j_{r+1} \leq n} \big(
\gamma_{j_1} z_{j_2} \dots z_{j_{r+1}}
 + (z_{j_1} + \gamma_{j_1}) {P}_{j_2, \dots, j_{r+1}} \big) \\&+ \sum_{k < j_1
< j_2 < \dots < j_r \leq n} \sum_{1 < t \leq k} \gamma_t \big(
\gamma_{j_1} z_{j_2} \dots z_{j_r} + (z_{j_1} + \gamma_{j_1})
P_{j_2, \dots ,j_r}  \big) \big).
\end{align*}
Now by (\ref{Peqn0}), this equals
\begin{align*}
\sum_{k < j_1 < \dots < j_r \leq n} z_1 z_{j_1}  \dots z_{j_{r}} -
\sum_{k < j_1 < \dots < j_r < j_{r+1} \leq n} {P}_{j_1, \dots,
j_{r+1}} + \sum_{k < j_1 < \dots < j_r \leq n}\sum_{1 < t \leq k}
\gamma_t P_{j_1, \dots ,j_r},
\end{align*}
as required.
\end{proof}

\subsection{Proof of Theorem \ref{centralthm}} Recall that we want
to show that $[x_1, \mathfrak{S}_r] \in \underline{t}H_{R}$ for any
$r$.

We calculate
\begin{align} [x_1, \mathfrak{S}_r] &= [x_1, \sum_{1\leq j_1 < j_2 < \dots < j_r \leq n}
z_{j_1} \dots z_{j_r}] \notag \\&= [x_1, \sum_{1 = j_1 < j_2 < \dots
< j_r \leq n} z_{j_1} \dots z_{j_r}] + [x_1,
\sum_{1 < j_1 < j_2 < \dots < j_r \leq n} z_{j_1} \dots z_{j_r}] \notag \\
&= \sum_{1 < j_2 < \dots < j_r \leq n} \left( [x_1,z_1] z_{j_2}
\dots z_{j_r} + z_1 [x_1, z_{j_2} \dots z_{j_r}] \right) + \sum_{1 <
j_1 < j_2 < \dots < j_r \leq n} P_{j_1, \dots ,j_r}x_1 \label{mainproof1}\\
&= \sum_{1 < j_2 < \dots < j_r \leq n} [x_1,z_1] z_{j_2} \dots
z_{j_r} + \sum_{1 < j_2 < \dots < j_r \leq n}z_1 P_{j_2, \dots
,j_r}x_1 + \sum_{1 < j_1 < j_2 < \dots < j_r \leq n} P_{j_1, \dots
,j_r}x_1. \label{mainproof2}
\end{align} The lines (\ref{mainproof1}) and (\ref{mainproof2})
follow from Lemma \ref{commrelations}.

By (\ref{zcomm}) and Lemma \ref{commrelations},
\begin{align*} \sum_{1 < j_2 < \dots < j_r \leq n} [x_1,z_1] z_{j_2} \dots
z_{j_r} &= \sum_{1 < j_2 < \dots < j_r \leq n} -\underline{t} x_1
z_{j_2} \dots z_{j_r} - \sum_{1 < j_2 < \dots
< j_r \leq n} \sum_{1 < t \leq n} \gamma_{t} x_1 z_{j_2} \dots z_{j_r}\\
&= \sum_{1 < j_2 < \dots < j_r \leq n} -\underline{t} x_1 z_{j_2}
\dots z_{j_r} - \sum_{1 < j_2 < \dots < j_r \leq n} \sum_{1 < t \leq
n} \gamma_{t}\tilde{P}_{j_2, \dots ,j_r}x_1.
\end{align*} Plugging this back into (\ref{mainproof2}), it suffices to show that
\begin{align}\label{mainproof3} \sum_{1 < j_2 < \dots < j_r \leq n} \sum_{1 < t \leq
n} - \gamma_{t} \tilde{P}_{j_2, \dots ,j_r} + \sum_{1 < j_2 < \dots
< j_r \leq n}z_1 P_{j_2, \dots ,j_r} + \sum_{1 < j_1 < j_2 < \dots <
j_r \leq n} P_{j_1, \dots ,j_r} = 0.
\end{align} Equivalently, we want to show that
\begin{align}\label{mainproof4} \sum_{1 < j_2 < \dots < j_r \leq n}(z_1 - \sum_{1 < t \leq
n} \gamma_{t}) \tilde{P}_{j_2, \dots ,j_r} - \sum_{1 < j_2 < \dots <
j_r \leq n}z_1 z_{j_2} \dots z_{j_r} + \sum_{1 < j_1 < j_2 < \dots <
j_r \leq n} P_{j_1, \dots ,j_r} =0.
\end{align} This is the statement of Proposition \ref{greatprop} (in
the case $k=1$).

\subsection{An alternative proof of Theorem \ref{centralthm}}\label{Jackproof}
We sketch now a proof of Theorem \ref{centralthm} using the theory
of Jack polynomials. We use freely the notation from \cite{Gr1}. Let
$W=G(m,1,n)$ and let $H_{R}=H_{R}(W)$. Let $\blambda = ((n),
\emptyset, \dots ,\emptyset)$ be the $m$-multipartition
corresponding to the trivial representation of $W$, see
\ref{irrepsW} below. Then $H_R$ has the standard polynomial
representation $\Delta_{R}(\blambda) = S_R(V^*)$. It is a faithful
representation of $H_R$. Indeed, if we consider $\Delta_K(\blambda)
:= K \otimes_R H_R$, which contains $H_R$ as a subring, then we
obtain a representation $K \otimes_R \Delta(\blambda)$ of $H_K$, and
this has a non-degenerate contravariant form. Therefore
$\Delta_K(\blambda)$ and also $\Delta_R(\blambda)$ are faithful.

There is a basis $f_\mu$ of $\Delta_K(\blambda)$ consisting of
non-symmetric Jack polynomials, where $\mu \in \Z^n$; see
\cite[Section 6]{Gr1}. The Jack polynomials do not have poles at
$\underline{t} = 0$. So in calculations involving $f_\mu$ it makes
sense to specialize $\underline{t}=0$ (this is not necessarily true
for the polynomials $f_{\mu,T}$ in other standard modules,
\cite{Gr2}).

In the notation of the proof of Lemma \ref{HeckeRelations}, we have
\begin{align*}z_i& =\tilde{z}_i -\frac{1}{2}\underline{t} - \sum_{l=1}^{m-1} \eta^{-l} \underline{c}_l g_i^l\\
&=\tilde{z}_i -\frac{1}{2}\underline{t} -\sum_{0 \leq l \leq m-1}
d_{l-1} \pi_{i,l},\end{align*} where the idempotents $\pi_{i,l}$ are
given by
\[\pi_{i,l}=\frac{1}{m} \sum_{0 \leq j \leq m-1} \eta^{-jl} g_i^j,\] and
\[ d_{l-1} = \sum_{k=1}^{m-1} \eta^{k(l-1)}\underline{c}_k. \]  Then by \cite[Theorem 6.1]{Gr1},
$$\pi_{i,l} f_\mu=\begin{cases} f_\mu \quad \hbox{if $l=-\mu_i$ and} \\ 0 \quad \hbox{else.}
\end{cases}$$

It follows from the above and \cite[Theorem 6.1]{Gr1} that $z_i$ acts on $f_\mu$ as
\begin{equation}
{z}_i.f_\mu=\left(\underline{t}(\mu_i+ \frac{1}{2})-d_0-m
(w_\mu(i)-1) \underline{\kappa} \right) f_\mu.
\end{equation} Working modulo $\underline{t}$, for any
$\mu,\nu \in \ZZ_{\geq 0}^n$ we have an equality of multisets:
$$\{\underline{t}(\mu_i+\frac{1}{2})-d_0-m (w_\mu(i)-1) \underline{\kappa} \}_{1 \leq i \leq n}=\{ \underline{t}(\nu_i+\frac{1}{2})-d_0-m (w_\nu(i)-1) \underline{\kappa} \}_{1 \leq i \leq n} \ \mathrm{mod} \ \underline{t}.$$

Therefore for any symmetric polynomial $f=f({z}_1,\dots,{z}_n)$ in
the operators ${z}_1,\dots,{z}_n$ the eigenvalues by which $f$ acts
on $f_\mu$ and $f_\nu$ are equal modulo $\underline{t}$. Applying
this with $\nu=\phi.\mu$ and $\nu=\phi^{-1}.\mu$ where $\phi$ is
defined in equation~(5.6) of \cite{Gr1}, it follows that
\begin{equation}
[\Phi,f].f_\mu=0 \ \mathrm{mod} \ \underline{t} \quad \text{and}
\quad [\Psi,f].f_\mu=0 \ \mathrm{mod} \ \underline{t} \quad
\hbox{for all $\mu \in \ZZ_{\geq 0}^n$.}
\end{equation}  Each monomial $x^\mu$ is an linear combination
of $f_\mu$, so we have also
\begin{equation}
[\Phi,f].x^\mu=0 \ \mathrm{mod} \ \underline{t} \quad
\text{and} \quad [\Psi,f].x^\mu=0 \ \mathrm{mod} \ \underline{t}
\quad \hbox{for all $\mu \in \ZZ_{\geq 0}^n$.}
\end{equation}  Since the polynomial representation is
faithful the relations $[\Phi,f]=[\Psi,f]=0 \ \mathrm{mod} \
\underline{t}$ hold in $H_{R}$.  But $w f=f w$ for all $w \in W$ by
\cite[Lemma 5.1]{Gr1}, so since $H_R$ is generated by $\Phi$,
$\Psi$, and $W$ it follows that $f$ becomes central upon setting
$\underline{t}=0$.

\medskip

\noindent \textbf{Remark.}  More generally, $z_i$ acts on the basis
$\{f_{\mu,T} \}$ of the standard module $\Delta_{K}(\blambda')$
constructed in \cite{Gr2} by
\begin{equation}
z_i.f_{\mu,T}=\left(\underline{t}(\mu_i+1)-d_{\beta(T^{-1}
w_\mu(i))}-m \mathrm{ct}(T^{-1} w_\mu(i)) \underline{\kappa} \right)
f_{\mu,T}.
\end{equation}  Setting $\underline{t}=0$ gives the eigenvalues by
which symmetric polynomials in the $z_i$'s will act on the baby
Verma module for $\blambda'$, even though $f_{\mu,T}$ may have a
pole at $\underline{t}=0$. This is written out in \ref{Zchar}.

\section{Blocks of $\overline{H}_c$ for $G(m,d,n)$}\label{blockresults}

\subsection{Jucys-Murphy elements for wreath products} Let
$W=G(m,1,n)$. We recall a generalisation of the beautiful
construction of representations of the symmetric group described in
\cite{OV}, which is explained in \cite{P}. Let $GZ_n$ be the
subalgebra of $\C W$ generated by elements $g_j^l$ for $1 \leq j
\leq n$ and $0\leq l \leq m-1$, and
\[u_i := \sum_{l=0}^{m-1} \sum_{1 \leq j < i} s_{ij} g_i^{-l} g_j^{l}\] for $1 \leq i \leq n$.
The algebra $GZ_n$ is a maximal commutative subalgebra of $\C W$ which acts
diagonally on the irreducible representations of $W$.

\subsection{Combinatorics}\label{combinatorics} To describe the eigenvalues of
$GZ_n$ on irreducible representations we need to introduce some
combinatorics, we follow the conventions of \cite[Section 3]{M}. We
denote by $\mathcal{P}(m,n)$ the set of $m$-multipartitions of $n$,
$\mathcal{P}(m,n) := \{ (\lambda^0, \dots ,\lambda^{m-1}):
\sum_{i=0}^{m-1} |\lambda^i| = n\}$. We will often identify a
mulitpartition with its $m$-tuple of Young diagrams. Given an
$m$-multipartition $\blambda = (\lambda^0, \dots ,\lambda^{m-1})$ we
define a \textit{standard tableau on} $\blambda$ to be a numbering
of the Young diagram of $\blambda$ with $\{ 1, \dots n\}$ such that
the numbers in the Young diagram corresponding to each $\lambda^i$
are row increasing and column increasing. One can think of a tableau
as a map $T: \{1, \dots, n\} \to \mathrm{Boxes\ of}\ \blambda.$ For
example, this is a standard tableau on the multipartition $\big(
(3,3), (2,1,1) \big)$:
\begin{align*} \Bigg( \ \Yvcentermath1 \young(\x \xxx \xxxxx,\xxxx
\xxxxxxxxx \xxxxxxxxxx)\ ,\ \young(\xx \xxxxxxxx,\xxxxxx,\xxxxxxx)\
\Bigg).
\end{align*}

Let $b$ be a box in $\blambda$. We denote by $\beta(b)$ the $i$ such
that $b$ lies in $\lambda^i$. The \textit{content} of a box $b$ in
$\blambda$ is the number $ct(b) := j-i$, where $\beta(b)=k$ for some
$k$ and $b$ lies in column $j$ and row $i$ of $\lambda^k$. The
\textit{residue} of a partition $\lambda^i$ is the polynomial
$\mathrm{Res}_{\lambda^i}(x) := \sum_{b \in \lambda^i}
x^{\mathrm{ct}(b)}$.

\subsection{The irreducible representations of $G(m,1,n)$}\label{irrepsW}
The irreducible representations of $W$ are labelled by the set
$\mathcal{P}(m,n)$. To each $\blambda \in \mathcal{P}(m,n)$ we
denote by the $V_{\blambda}$ the corresponding representation. By
\cite[Theorem 9]{P}, there exists a basis $\{ v_T:\ \mathrm{T\ is\
a\ standard\ tableau\ on\ \blambda} \}$ of $V_{\blambda}$ such that
the $v_T$ are eigenvectors for $GZ_n$. Let us describe the
eigenvalues. Let $T$ be a standard tableau on $\blambda$ and let
$\beta_T(i) = \beta(T(i))$. Then $u_i$ and $g_j^l$ act on $v_T$ by
scalars $m. ct(T(i))$ and $\eta^{l \beta_T(i)}$, respectively.

\subsection{Characters of $\mathfrak{Z}_{c}$}\label{Zchar} Let ${c}=
(0,\kappa, c_1, \dots ,c_{m-1}) \in \C^{m+1}$. Recall the notation
of \ref{RRCA} and \ref{babyVs}. We calculate the character of
$\mathfrak{Z}_{{c}} := {R}/m_{(0,c)} \otimes_{R} \mathfrak{Z}_{R}$
on a baby Verma module $\overline{\Delta}_{c}(\blambda)$. Let us
first note that $\mathfrak{Z}_{{c }}$ acts on $\overline{\Delta}_{{c
}}(\blambda)$ by a scalar. This is because $\mathfrak{Z}_{{c}}$ acts
on the subspace $1 \otimes V_{\blambda} \subset
\overline{\Delta}_{{c}}(\blambda)$, and since $\mathfrak{Z}_{{c}}$
is central (and in particular commutes with $W$) its acts on $1
\otimes V_{\blambda}$, and therefore also on
$\overline{\Delta}_{{c}}(\blambda)$, via a scalar.

Let $T$ be a standard tableau on $\blambda$. Since $y_i$ annhilates
the subspace $1 \otimes V_{\blambda}$, $z_i \in \mathfrak{Z}_{c}$
acts on $1 \otimes V_{\blambda}$ via the element \[ - {{\kappa}}
\sum_{l=0}^{m-1} \sum_{i < j} s_{ij} g_i^{-l} g_j^l -
\sum_{l=1}^{m-1} {{c}}_l g_i^l .\] This equals \[ w_0
(-\underline{\kappa}u_{n-i+1} - \sum_{l=1}^{m-1} {{c}}_l
g_{n-i+1}^l)w_0,\] where $w_0 \in S_n$ denotes the longest element.
Thus the set of $\{ w_0 \cdot (1 \otimes v_T)\}$ forms a basis of
eigenvectors for the action of the $z_i$. By \ref{irrepsW},
$z_{n-i+1}$ acts on $w_0 \cdot (1 \otimes  v_T)$ via the eigenvalue
\begin{align}\label{evaluesVermas} - {{\kappa}} m. {ct}(T(i)) -
\sum_{l=1}^{m-1} {{c}}_l \eta^{l \beta_{T}(i)}.\end{align} Recall
the parameters ${H}_i$ from \ref{wreathcherednik} and the formula
relating the ${c}_i$ and ${H}_i$, (\ref{parameters}).
Let us fix $i$ and set $\beta = \beta_{T}(i)$, then
\begin{align} - \sum_{l=1}^{m-1} {c}_l \eta^{\beta l} = \sum_{l=1}^{m-1} \sum_{j=0}^{m-1}
\eta^{\beta l} \frac{\eta^{-lj}}{(1-\eta^{-l})} {H}_j & =
\sum_{j=1}^{m-1} \sum_{l=1}^{m-1} \eta^{\beta l}
\frac{(\eta^{-lj}-1)}{(1-\eta^{-l})} {H}_j\notag \\ &=
-\sum_{j=1}^{m-1} \left( \sum_{l=1}^{m-1} \eta^{\beta l} (1 +
\eta^{-l} + \dots + \eta^{-l(j-1)}) \right) {H}_j\notag \\
&= \sum_{j=1}^{\beta} j
{H}_j + \sum_{j=\beta + 1}^{m-1} (j - m) {H}_j\notag \\
& = \sum_{j=1}^{m-1} (j-m) {H}_j + \sum_{j=1}^{\beta} m
{H}_j.\label{parametercalc}
\end{align}
Let $C = \sum_{j=1}^{m-1} (j-m) {H}_j$, and let ${a} = ( 0, {H}_1,
{H}_1 + {H}_2, \dots , {H}_1 + \dots + {H}_{m-1})$.

Since $\mathfrak{Z}_{{\bf c }}$ is generated by symmetric
polynomials in the $z_i$, the character of $\mathfrak{Z}_{{c}}$ on
$\overline{\Delta}_c(\blambda)$ is given by the unordered $n$-tuple
of the eigenvalues from (\ref{evaluesVermas}). It is convenient to
replace the term $- \sum_{l=1}^{m-1} {c}_l \eta^{\beta l}$ with
(\ref{parametercalc}) and to express the $n$-tuple using
polynomials:
\begin{align}\label{residueschar} \sum_{i=1}^n
 x^{\sum_{l=1}^{m-1} -{c}_l \eta^{\beta_{T(i)}l}}
 x^{- {\kappa} m ct(T(i))} = x^C \sum_{j=0}^{m-1} x^{m{a}_j} \Res_{\lambda^j} (x^{-m {\kappa} }). \end{align}
%
Finally, (\ref{residueschar}) is uniquely determined by the
expression: \begin{align}\label{residuechar2} \sum_{j=0}^{m-1}
x^{{a}_j} \Res_{\lambda^j} (x^{{-\kappa}}).
\end{align}

\subsection{Blocks for wreath products}\label{wreathblocks}

\begin{thm}
Let $c = (0,\kappa, c_1, \dots ,c_{m-1}) \in \C^{m+1}$ be a
parameter as in \ref{wreathcherednik}. Let $\blambda, \bmu \in
\mathcal{P}(m,n)$. Let $a = (0 , H_1, H_1 + H_2, \dots , H_1 + \dots
+ H_{m-1})$. Then the baby Verma modules
$\overline{\Delta}_c(\blambda), \overline{\Delta}_c(\bmu)$ lie in
the same block if and only if $\sum_{i=0}^{m-1} x^{{a}_i}
\Res_{\lambda^i}(x^{-\kappa}) = \sum_{i=0}^{m-1} x^{{a}_i}
\Res_{\mu^i}(x^{-\kappa})$.
\end{thm}
\begin{proof}
Recall that there is a a grading on $H_{c}$ where the $y_i$ are in
degree $1$, the $x_i$ in degree $-1$, and $W$ is in degree $0$. Let
$Z_{\bf c }$ denote the centre of $H_{c}$, a graded subalgebra of
$H_c$. Let $f_1, \dots ,f_k$ be a set of homogeneous algebra
generators for $S[V \oplus V^*]^W$. Since $\gr Z_{c} = S[V \oplus
V^*]^W$ is a domain, any lifts $\tilde{f}_i$ of the $f_i$ form a set
of algebra generators for $Z_{c}$. Furthermore, since the filtered
pieces of $Z_{c}$ are graded, \ref{PBW}, we can choose the
$\tilde{f}_i$ to be homogeneous of the same degree as the $f_i$.

We begin by finding a set of homogeneous generators for $S[V \oplus
V^*]^W$. This amounts to finding generators for \[ \left( \C [x_i,
y_j: 1 \leq i,j \leq n]^{C_m^n} \right)^{S_n} = \C [x_i^m, x_jy_j,
y_k^m: 1 \leq i,j,k \leq n]^{S_n}. \] By Weyl's Theorem,
\cite{Weyl}, a generating set is given by the power sums
\begin{align} P_{q,r,s} := \sum_{ 1 \leq i \leq n} (x_i^m)^q
(x_iy_i)^r (y_i^m)^s , \end{align} where $q,r$ and $s$ run through
all values such that $1 \leq q+r+s \leq n$. The degree of
$P_{q,r,s}$ equals $-mq + ms$, so the only degree zero generators in
this set are the $P_{q,r,q}$. A lift of $P_{q,r,q}$ is given by
$\tilde{P}_{q,r,q} := \sum_{1 \leq i \leq n} z_i^{mq + r} \in
\mathfrak{Z}_{{c}}$. Let $\tilde{P}_{q,r,s}$ denote homogeneous
lifts of the remaining generators.

By M\"{u}ller's Theorem, \cite[Theorem III.9.2]{BG}, the blocks of
$\overline{H}_c$ are determined by the characters of $Z_c$ on baby
Verma modules. For any generator $\tilde{P}_{q,r,s}$ of degree not
equal to zero, we know that
$\tilde{P}_{q,r,s}|_{\overline{\Delta}(\blambda)}$ is a nilpotent
operator, since the action of $\tilde{P}_{q,r,s}$ factors through
the finite-dimensional $\Z$-graded algebra $\overline{H}_c$. Thus
$\tilde{P}_{q,r,s}|_{\overline{\Delta}(\blambda)} = 0$. It follows
that $\overline{\Delta}_c(\blambda), \overline{\Delta}_c(\bmu)$ lie
in the same block if and only if the characters of
$\tilde{P}_{q,r,q}$ on these modules are equal. Since
$\C[\tilde{P}_{q,r,q}] = \mathfrak{Z}_{{c}}$, this is the same as
comparing characters for the latter algebra. This was calculated in
the previous section.

\end{proof}

\noindent \textbf{Remark.} If $\kappa \neq 0$, then we can assume
without loss of generality that $\kappa = -1$, see (\ref{shiftiso}).
Thus $\overline{\Delta}_c(\blambda), \overline{\Delta}_c(\bmu)$ lie
in the same block if and only if \[\sum_{i=0}^{m-1} x^{{a}_i}
\Res_{\lambda^i}(x) = \sum_{i=0}^{m-1} x^{{a}_i} \Res_{\mu^i}(x).\]
This is precisely the condition obtained in \cite[Theorem 3.13]{M}
for rational $H_i$.

In the case $\kappa = 0$, we get that $\overline{\Delta}_c(\blambda), \overline{\Delta}_c(\bmu)$ lie in
the same block if and only if $\sum_{i=0}^{m-1} x^{{a}_i}
|\mu^i| = \sum_{i=0}^{m-1} x^{{a}_i}|\lambda^i|$.

\subsection{Irreducible modules for $G(m,d,n)$}\label{normalirreps}

Let $d$ be a positive integer which divides $m$, and let $p=
\frac{m}{d}$. We assume throughout that $n > 2$, or $n=2$ and $d$ is
odd - this avoids degeneration of conjugacy classes below. We
describe a parametrisation of irreducible modules for the normal
subgroup $G(m,d,n)$ of $G(m,1,n)$. See \cite[Section 5]{B} for a
detailed discussion of the group $G(m,d,n)$. Let $\delta$ be a
generator of the cyclic group $C_d$. We define an action of $C_d$ on
$\mathcal{P}(m,n)$ by \bdm \delta \cdot (\lambda^0, \dots,
\lambda^{m-1}) = (\lambda^{m - p}, \lambda^{m+1-p}, \dots
,\lambda^{m-2},\lambda^{m-1},\lambda^0,\lambda^1, \dots,
\lambda^{m-p-1}). \edm Let $C_{\blambda}$ be the stabiliser of
$\blambda$ under this action, and let $\{\blambda \}$ denote the
orbit of $\blambda$. The irreducible representations of $G(m,d,n)$
are parameterized by distinct pairs $(\{ \blambda \},\epsilon)$,
where $\blambda \in \mathcal{P}(m,n)$ and $\epsilon \in
C_{\blambda}$.

The conjugacy classes of reflections in $G(m,d,n)$ are given by:
\begin{align*} \{s_{ij} g_i^{-l} g_j^l: 0 \leq l \leq m-1\ \mathrm{and}\ i\neq j\},
\end{align*} and, for each $1 \leq l \leq p-1$, \begin{align*}\{
g_j^{dl}: 1 \leq j \leq n\}.\end{align*} In particular, a parameter
for the rational Cherednik algebra of $G(m,d,n)$ is given by a tuple
$c' = (0,\kappa, c_d, c_{2d}, \dots ,c_{d(p-1)}) \in \C^{p+1}$. If
we set $c = (0,\kappa, c_1, \dots ,c_{m-1})$, where $c_l = 0$ if $d
\nmid l$, then there is an algebra monomorphism \[ H_{c'}(G(m,d,n))
\hookrightarrow H_c(G(m,1,n)).\]

\subsection{Blocks for $G(m,d,n)$}\label{normalblocks}
Let us write $\blambda = (\underline{\lambda}_0, \dots ,
\underline{\lambda}_{d-1})$ where $\underline{\lambda}_i =
(\lambda^{ip}, \dots , \lambda^{(i+1)p - 1})$. A multipartition
$\blambda$ is called \textit{$d$-stuttering} if
$\underline{\lambda}_i = \underline{\lambda}_j$ for all $0 \le i,j
\le d-1$. The following theorem is \cite[Corollary 6.10]{B}, whose
proof remains valid for arbitrary $c'$ once Theorem
\ref{wreathblocks} is known.

\begin{thm}
Let $c'$ and $c$ be as in \ref{normalirreps}. Let $( \{ \blambda \},
\epsilon) , ( \{ \bmu \}, \eta)$ be pairs labelling irreducible
$G(m,d,n)$-modules. Then:
\begin{itemize}
\item if $\{ \blambda \} \neq \{ \bmu \}$, then $\overline{\Delta}_{c'}( \{ \blambda \}, \epsilon)$ and  $\overline{\Delta}_{c'}( \{ \bmu \}, \eta)$ lie in the same block if and only if
\bdm
\sum_{i=0}^{m-1} x^{{a}_i}
\Res_{\lambda^i}(x^{-\kappa}) = \sum_{i=0}^{m-1} x^{{a}_i}
\Res_{\mu^i}(x^{-\kappa});
\edm
\item  if $\blambda$ is a $d$-stuttering multipartition and $$\sum_{i=0}^{m-1} x^{{a}_i}
\Res_{\lambda^i}(x^{-\kappa}) \neq \sum_{i=0}^{m-1} x^{{a}_i}
\Res_{\mu^i}(x^{-\kappa})$$ for all $\blambda \neq \bmu \in
\mc{P}(m,n)$, then $\overline{\Delta}_{c'}( \{ \blambda \},
\epsilon)$ and  $\overline{\Delta}_{c'}( \{ \blambda \}, \eta)$ are
in the same block if and only if $\epsilon = \eta$;
\item otherwise, $\overline{\Delta}_{c'}( \{ \blambda \}, \epsilon)$ and  $\overline{\Delta}_{c'}( \{ \blambda \}, \eta)$ are in the same block.
\end{itemize}
\end{thm}

\bibliographystyle{alphanum}
\bibliography{GradedHecke}

\end{document}